\documentclass{amsart}

\usepackage[pagewise]{lineno}
\usepackage{graphicx}
\usepackage{amssymb}
\usepackage{enumerate}
\newtheorem{theorem}{Theorem}[section]
\newtheorem{lemma}[theorem]{Lemma}

\theoremstyle{definition}
\newtheorem{definition}[theorem]{Definition}
\newtheorem{example}[theorem]{Example}

\newtheorem{prop}{Proposition}[section]

\theoremstyle{remark}
\newtheorem{remark}{\bf{Remark}}[section]

\numberwithin{equation}{section}

\begin{document}

\title{ON CHAOTIC BEHAVIOR OF  ASH ATTRACTORS}

\author{Elias Rego}
\address{Department of Mathematics, Southern University of Science and Technology, Guangdong  Shenzhen, China.}
\email{elias@im.ufrj.br}

\author{Kendry J. Vivas}
\address{Instituto de Matem\'aticas, Pontificia Universidad Cat\'olica de Valpara\'iso, Valpara\'iso, Chile.}
\email{kendry.vivas01@ucn.cl}
\thanks{The second author was supported by ANID/FONDECYT Postdoctorado/3220009, Chile.}




\begin{abstract}
The asymptotic sectional hyperbolicity is a weak notion of hyperbolicity that extends properly the sectional-hyperbolicity and includes the Rovella attractor as a archetypal example. The main feature of this definition is the existence of arbitrarily large hyperbolic times for points outside the stable manifolds of the singularities. In this paper we will prove that any attractor associated to a $C^1$ vector field $X$ on a three-dimensional manifold satisfying this kind of hyperbolicity is rescaling expansive and presents sensitiveness respect to initial conditions.  
\end{abstract}

\maketitle

\section{Introduction}

The notions of expansiveness and sensitiveness have been widely studied in dynamical systems since it occurs in dynamics exhibiting chaotic behavior. In simple terms, these properties allows us to distinguish orbits within a dynamical system. For instance, in systems with discrete time, expansiveness means that the orbits of different points eventually are uniformly separated at some iterate $n$, while for continuous flows the definition is more complicated.  

The most well-known definition of expansiveness for flows was introduced by Bowen and Walters in \cite{BW} by using time-reparametrizations, which is usually named $BW$-\textit{expansiveness}. They showed that hyperbolic sets are examples of systems satisfying this definition. Besides, the singularities of systems with this property must been isolated from regular orbits. This shows, in particular, that the geometric Lorenz attractor (GLA), constructed in an independent way by Afraimovich, Bykov and Shilnikov and Guckenheimer and \cite{ABS} in \cite{G} respectively, is not $BW$-expansive, so it was necessary to have a suitable definition of expansiveness both for this set and for other singular flows. In this way, Komuro in \cite{K} introduced the notion of $\mathcal{K}^*$-\textit{expansiveness}, which is equivalent to $BW$-expansiveness for non-singular flows, and showed that the GLA satisfied this dynamical property. On the other hand, in \cite{Mo3} was showed that this attractor is an example of systems satisfying the notion of singular hyperbolcity. Motivated by this result, the authors in \cite{AP} proved that every singular-hyperbolic attractor associated to a $C^1$ vector fields on a three-dimensional manifold also satisfy this kind of expansiveness. Moreover, they showed, as a direct consequence, that these systems are also sensitive to initial conditions. To our best knowledge, it is still not known if singular hyperbolicity implies $\mathcal{K}^*$-expansiveness in higher dimensional manifolds.

Now, inspired in the work of Liao about standard systems of differential equations \cite{L1},\cite{L2}, X. Wen and L. Wen in \cite{WW} introduced a new notion of expansiveness which is called \textit{rescaled-expansiveness} (or $R$-expansiveness for short). The main feature of these systems is to made a rescaling in the size of the tubular neighborhood of a regular orbit by the size of its velocity vector field. Among the main results of \cite{WW} it was shown that singular-hyperbolic systems are examples of $R$-expansive systems in any dimension.

At this point it is worth to mention that despite the similarity in the definitions  of $\mathcal{K}^*$-expansiveness and $R$-expansiveness  (see section 2), these concepts are in fact distinct. Indeed, the trivial identity flow and the GLA are examples of $R$-expansive flows, but only the GLA is $\mathcal{K}^*$-expansive (See \cite{WW}). This illustrates that the set of $R$-expansive systems contains properly the singular-hyperbolic systems. On the other hand, $R$-expansiveness and $\mathcal{K}^*$-expansiveness are closely related properties for flows whose singularities exhibit nice structure. For precise information about these similarities we reefer the reader to the work  \cite{A}, where it is studied the relation between $R$-expansiveness,  $\mathcal{K}^*$-expansiveness and the separating property.

  According to the results given by Artigue in \cite{A} and by Carrasco and San Mart\'in in \cite{CSM}, it is possible to show that the Rovella attractor is an example of a system with hyperbolic singularities that is $R$-expansive, but it is not singular-hyperbolic. On the other hand, in \cite{SMV2} it was showed that this attractor is a prototype example of systems that are Asymptotically Sectional-Hyperbolic (ASH), which is a weak notion of hyperbolicity that extends properly the singular-hyperbolicity for three-dimensional vector fields and whose main feature is the existence of arbitrarily large hyperbolic times in every point outside the stable manifolds of its singularities. In this way, it is natural to ask if any three-dimensional asymptotically sectional-hyperbolic attractor satisfies this dynamical property. In this paper we will give an affirmative answer to this question for $C^1$ vector fields. Besides, we will prove that ASH attractors are also sensitive to initial data. As we will see, this property is not obtained as a direct consequence of $R$-expansiveness, so that we will present a separate proof for this result.

\section{Statements of the results}
Throughout this paper we denote by $M$ to a compact Riemannian manifold endowed with the metric $d$ induced by the Riemannian metric $\Vert\cdot\Vert$. By a {\it{flow}} on $M$ we refer to the one-parameter family of maps $X_t$, $t\in\mathbb{R}$, induced by a vector field $X$ on $M$ of class $C^k$, $k\geq 1$. For $x\in M$, by an {\it{orbit}} of $x$ we will refer to the set $\mathcal{O}(x)=\lbrace X_t(x) : t\in\mathbb{R}\rbrace$. Given an interval $[a,b]\subset\mathbb{R}$ and a subset $A\subset M$, we set $X_{[a,b]}(A)=\bigcup_{t\in [a,b]}X_t(A)$. By $Sing(X)$ we will denote the set of singularities (i.e. zeros of $X$). 
By \textit{regular orbit} we mean to an orbit which is not associated to singularities. We will say that a subset $\Lambda$ of $M$ is {\it{invariant }} if $X_t(\Lambda)=\Lambda$  for all $t\in\mathbb{R}$. Recall that that a compact invariant set $\Lambda$ is {\it{attracting}} if there exists a neighborhood $U$ of $\Lambda$ (called trapping region) such that $X_t(\overline{U})\subset U$ for any $t>0$ and
\begin{displaymath}
\Lambda=\bigcap_{t\geq 0}X_t(U).
\end{displaymath}
An {\it{attractor}} is a set which is attracting and transitive, i.e, there is $z\in\Lambda$ such that $\omega(z)=\Lambda$.   

Now, a compact invariant set $\Lambda$ has a {\it{dominated splitting}} if there are a continuous invariant splitting $T_{\Lambda}M=E\oplus E^c$ and positive numbers $K$ and $\lambda$ such that 
\begin{displaymath}
\frac{\Vert DX_t(x)\vert_{E_x}\Vert}{m(DX_t(x)\vert_{E^c_x})}\leq Ke^{-\lambda t},\quad\forall x\in\Lambda,\forall t>0.
\end{displaymath} 
Here, $m(A)$ denotes the minimum norm of the linear map $A$. In this case we say that $E^c$ is {\it{dominated}} by $E$. If $E$ is uniformly contracting, i.e., $\Vert DX_t(x)\vert_{E_x}\Vert\leq Ke^{-\lambda t}$ for every $t>0$ and $x\in\Lambda$, we say that $\Lambda$ is \textit{partially hyperbolic}. 

With the purpose to extend the notion of hyperbolicity to include the geometric Lorenz attractor as a prototype example, the authors in \cite{Mo3} introduced the notion of {\it{singular-hyperbolic}} set as a compact invariant partially hyperbolic set $\Lambda$ whose singularities are hyperbolic for which the central subbundle $E^c$ is {\it{volume expanding}}, i.e., there are positive numbers $K,\lambda$ such that 
\begin{displaymath}
\vert \det DX_t(x)\vert_{E^c_x}\vert\geq Ke^{\lambda t},\quad\forall x\in\Lambda,\forall t> 0.
\end{displaymath}
On the other hand, the notion of sectional hyperbolicity was introduced in \cite{Me3} to study the dynamic behavior of certain higher-dimensional flows. More precisely, an invariant compact partially hyperbolic set $\Lambda$ is \textit{sectional-hyperbolic} if all its singularities are hyperbolic and the central subbundle $E^c$ is {\it{sectional expanding}}, i.e, there are positive numbers $K,\lambda$ such that 
\begin{displaymath}
\vert \det DX_t(x)\vert_{L_x}\vert\geq Ke^{\lambda t},\quad\forall x\in\Lambda,\forall t> 0,
\end{displaymath}    
where $L_x$ is a two-dimensional subspace of $E^c_x$. The notions of sectional and singular hyperbolicity agree in three-dimensional manifolds, whereas in higher dimensions there exists vector fields with singular-hyperbolic sets which are not sectional-hyperbolic (see \cite{Me3}).  

On the other hand, a well known fact says that for any hyperbolic singularity $\sigma$ of $X$, its stable and unstable sets $W^s(\sigma)$, $W^u(\sigma)$ are manifolds of class $C^k$, which are tangent to $E_{\sigma}^s$ and $E_{\sigma}^u$ at $\sigma$  respectively. Let   
\begin{displaymath}
W^s(Sing(X))=\bigcup_{\sigma\in Sing(X)}W^s(\sigma).
\end{displaymath}
With this in mind, let us remember the notion of \textit{asymptotically sectional-hyperbolic set}, introduced in \cite{Mo4}: 

\begin{definition}
Let $\Lambda$ be an invariant compact partially hyperbolic set of a vector field $X$ whose singularities are hyperbolic.  We say that $\Lambda$ is {\it{asymptotically sectional-hyperbolic}} (ASH for short) if the central subbundle is eventually asymptotically expanding outside the stable manifolds of the singularities, i.e, there exists $C>0$ such that 
\begin{equation}\label{ash}
\limsup_{t\to+\infty}\frac{\log\vert \text{det}(DX_t(x)\vert_{L_x})\vert}{t}\geq C,
\end{equation}
for every $x\in \Lambda'=\Lambda\setminus W^s(Sing(X))$ and every two-dimensional subspace $L_x$ of $E^c_x$.
\end{definition}

A direct consequence of ASH property is the existence of arbitrarily large ``hyperbolic times''. Namely, for every $x$ in $\Lambda'$ and every two-dimensional subspace $L_x$ of $E_x^c$ there is an unbounded increasing sequence of positive numbers $t_k=t_k(x,L_x)>0$ such that
\begin{equation}\label{ht}
\lim_{k\to\infty}t_k=\infty\quad\text{and }\quad\vert \det DX_{t_k}(x)\vert_{L_x}\vert\geq e^{Ct_k},\quad k\geq 1.
\end{equation}
A sequence that satisfies the above relation will be called $C$-\textit{hyperbolic times} for $x$.

\begin{remark}
It follows by definition that every singular-hyperbolic set is asymptotically sectional-hyperbolic but not conversely, i.e., the asymptotically sectional-hyperbolicity extends properly the notion of sectional-hyperbolicity. 
\end{remark}
\begin{remark}
In \cite{SMV2} was proved that every asymptotically sectional-hyperbolic set satisfies the Hyperbolic Lemma. 
\end{remark}

From now on we only consider $C^1$ vector fields $X$ on three-dimensional Riemannian manifolds $M$. Let consider $\sigma\in\text{Sing}(X)$. We say that $\sigma$ is \textit{attached} to $\Lambda$ if it is accumulated by regular orbits in $\Lambda$. On the other hand, we say that $\sigma$ is a \textit{real index two singularity} if it has three real eigenvalues satisfying  $\lambda_{ss}<\lambda_s<0<\lambda_u$. In this case, we denote by $W^{ss}(\sigma)$ to the strong stable manifold of $\sigma$ associated to $\lambda_{ss}$. A real index two singularity is called 
\begin{itemize}
\item \textit{Lorenz-like} if satisfies the central expanding condition $0<-\frac{\lambda_s}{\lambda_u}<1$.  
    \item {\it{Rovella-like}} if it satisfies the central contracting condition $-\frac{\lambda_s}{\lambda_u}>1$.
    \item {\it{Resonant}} if it satisfies the relation $-\frac{\lambda_s}{\lambda_u}=1$.
\end{itemize}
It is well known that any attached singularity $\sigma$ contained in a singular-hyperbolic connected sets $\Lambda$ is Lorenz-like and satisfies $\Lambda\cap W^{ss}(\sigma)=\lbrace\sigma\rbrace$ (see \cite{B}). Nevertheless, this is not the case in the context of asymptotic sectional-hyperbolicity, e.g. in \cite{SMV1} was exhibited a connected invariant set whose singularities are sinks. For ASH sets we have the following result whose proof is analogous to that given in Theorem A in \cite{Mo3}: 

\begin{theorem}\label{T1}
Let $\Lambda$ be a nontrivial asymptotically sectional-hyperbolic set of $X$ and assume that $\Lambda$ is not hyperbolic. Then, $\Lambda$ has at least one attached singularity. In addition, if $\Lambda$ is transitive, the following holds for $X$: Each singularity $\sigma$ of $\Lambda$ is either Lorenz-like, Rovella-like or resonant and satisfies 
\begin{equation}\label{C}
\Lambda\cap W^{ss}(\sigma)=\lbrace\sigma\rbrace.
\end{equation}
\end{theorem}

Now, the notion of expansiveness has been widely studied with the intention to understand the dynamic behavior of chaotic systems. In particular, for systems with continuous time, Bowen and Walters adapted the notion of expansiveness previously given for homeomorphisms by considering time-reparametrizations. Note that in a system with this kind of expansivity, the singularities are isolated from regular orbits. For systems with singularities accumulated by regular orbits, like the geometric Lorenz attractor for instance, Komuro in \cite{K} introduced the $\mathcal{K}^*$-expansiveness as follows: A flow $X_t$ on $M$ is $\mathcal{K}^*$\textit{-expansive} if for every $\varepsilon>0$ there is $\delta=\delta(\varepsilon)>0$ such that if $x,y\in M$ satisfy $d(X_t(x),X_{\theta(t)}(y))\leq\delta$ for every $t\in\mathbb{R}$ and some continuous, surjective and increasing function $\theta:\mathbb{R}\to\mathbb{R}$, there is $t_0\in\mathbb{R}$ such that $X_{\theta(t_0)}(y)\in X_{[t_0-\varepsilon,t_0+\varepsilon]}(x)$. Recently, X. Wen and L. Wen in \cite{WW} introduced the notion of rescaled-expansiveness inspired in the works of Liao (see \cite{L1} and \cite{L2}) as follows: 
\begin{definition}
A flow $X_t$ is \textit{rescaling expansive} ($R$-\textit{expansive} for short) on a compact invariant set $\Lambda$ if for every $\varepsilon>0$ there is $\delta=\delta(\varepsilon)>0$ such that, for any $x,y\in \Lambda$ and any increasing continuous function $\theta:\mathbb{R}\to\mathbb{R}$, if  $d(X_t(x),X_{\theta(t)}(y))\leq\delta\Vert X(X_t(x))\Vert$ for all $t\in\mathbb{R}$, then $X_{\theta(t)}(y)\in X_{[-\varepsilon,\varepsilon]}(X_t(x))$ for every $t\in\mathbb{R}$.
\end{definition}
In \cite{WY} the authors showed that the above definition is equivalent to the following: A flow $X_t$ is $R$-expansive on $\Lambda$ if for any $\varepsilon>0$, there is $\delta>0$ such that, for any $x,y\in \Lambda$ and an increasing homeomorphism $\theta:\mathbb{R}\to\mathbb{R}$, if  $d(X_t(x),X_{\theta(t)}(y))\leq\delta\Vert X(X_t(x))\Vert$ for all $t\in\mathbb{R}$, then $X_{\theta(0)}(y)\in X_{[-\varepsilon,\varepsilon]}(x)$.   

In \cite{AP} was showed that any singular-hyperbolic attractor associated to a $C^1$ vector field $X$ on a three-dimensional manifold is $\mathcal{K}^*$-expansive. In \cite{CSM} was proved that the Rovella atractor is an example of an attractor with Rovella singularities (so is not singular-hyperbolic) that satisfies this kind of expansiveness. So, by a result due to A. Artigue in \cite{A}, this attractor is $R$-expansive. Moreover, in \cite{SMV2} was proved that the Rovella attractor is a prototype example of systems which are ASH. With this in mind, in this paper we prove the following result: 

\begin{theorem}\label{TeoP}
Every asymptotically sectional-hyperbolic attractor $\Lambda$ associated to a $C^1$ vector field $X$ on $M$ is $R$-expansive.  
\end{theorem}

Finally, sensitive dependence on initial conditions is a very weak form of chaos that  describes, in simple terms, the way that orbits of nearby points are deviate. More precisely, we say that the flow of a vector field $X$ is \textit{sensitive to initial conditions} if there is $\delta>0$ such
that, for any $x\in M$ and any neighborhood $N$ of $x$, there is a point $y\in N$ and $t\geq 0$ such that 
\begin{displaymath}
d(X_t(x),X_t(y))>\delta. 
\end{displaymath}
The next theorem shows that ASH attractors also satisfy this property: 

\begin{theorem}\label{TeoSens}
Every asymptotically sectional-hyperbolic attractor associated to a $C^1$ vector field $X$ on $M$ is sensitive to initial conditions.
\end{theorem}

\section{Proofs}

In a similar way to the showed in \cite{AP}, the proof of Theorem 2.4 is by contradiction, i.e, suppose that there exists $\varepsilon>0$, a sequence $\delta_n\to0$, a sequence of increasing homeomorphisms $\theta_n:\mathbb{R}\to\mathbb{R}$, and a sequence of points $x_n, y_n\in\Lambda$ such that
\begin{equation}\label{assumption1}
d(X_t(x_n),X_{\theta_n(t)}(y_n))\leq\delta_n\Vert X(X_t(x_n))\Vert, \quad\forall t\in\mathbb{R},
\end{equation}
but 
\begin{equation}\label{assumption2}
X_{\theta_n(0)}(y_n)\notin X_{[-\varepsilon,\varepsilon]}(x_n).
\end{equation}

It should be noted that most of parts of the proof of Theorem 2.4 resemble that of the exhibited in the proof of Theorem A in \cite{AP}. In that reference, the crucial step in the argument was to explode the exponential growth of area in the central subbundle of every point in the attractor. Nevertheless, this is not the case for ASH sets, so that some previous results are needed. 

First, we recall the construction of adapted cross sections given in \cite{AP}: Let $E^s\oplus E^c$ be the partially hyperbolic splitting associated to the asymptotically sectional-hyperbolic attractor $\Lambda$, and let consider a continuous extension $\widetilde{E}^s\oplus \widetilde{E}^c$ to the basin of attraction $U$. In what follows, with the purpose to simplifying the notation, we write $E^{s,c}$ for $\widetilde{E}^{s,c}$. It is well known that $E^s$ can be chosen invariant by $DX_t$ for positive $t$. Nevertheless, the subbundle $E^c$ is not invariant in general, but it is possible to consider an invariant cone field $C_a^c$ of size $a>0$ around $E^c$ on $U$ defined by 
\begin{displaymath}
C_a^c(x):=\lbrace v=v_s+v_c : v_s\in E^s, v_c\in E^c\text{ and } \Vert v_s\Vert\leq a\Vert v_c\Vert\rbrace,\quad \forall x\in U.
\end{displaymath}
\begin{remark}
By shrinking $U$ if it is necessary, the number $a>0$ can be taken arbitrarily small.
\end{remark}

Now, by Proposition 2.1 in \cite{AP} we can obtain the local stable manifold $W_{\varepsilon}^s(x)$, $0<\varepsilon<1$, for every $x\in U$. Let $\Sigma'$ be a cross-section to $X$ containing $x$ in its interior. Define $W^s(x,\Sigma')$ to be the connected component of $W_{\varepsilon}^s(x)\cap\Sigma'$. This gives us a foliation $\mathcal{F}_{\Sigma'}$ of $\Sigma'$.  By Remark 2.2 in that article, we can construct a smaller cross section $\Sigma$, which is the image of a diffeomorphism $h:[-1,1]\times[-1,1]\to \Sigma'$, that sends vertical lines inside $\mathcal{F}_{\Sigma'}$ in a such way that $x$ belongs to the interior of $h([-1,1]\times[-1,1])$. In this case, the \textit{stable boundary} $\partial^s\Sigma$ and \textit{cu-boundary} $\partial^{cu}\Sigma$ of $\Sigma$ are defined by 
\begin{displaymath}
\partial^s\Sigma=h(\lbrace -1,1\rbrace\times[-1,1])\quad\text{ and }\quad\partial^{cu}\Sigma=h([-1,1]\times\lbrace-1,1\rbrace)
\end{displaymath}
respectively. We say that $\Sigma$ is $\eta$-\textit{adapted} if 
\begin{displaymath}
d(\Lambda\cap\Sigma,\partial^{cu}\Sigma)>\eta. 
\end{displaymath}

A consequence of the hyperbolic lemma (see \cite{SMV2}) is the following: 
\begin{prop}
Let $\Lambda$ be an asymptotically sectional-hyperbolic set attractor and let $x\in\Lambda$ be a regular point. Then,
there exists a $\eta_0$-adapted cross-section $\Sigma$ at $x$ for some $\eta_0>0$. 
\end{prop}

\begin{remark}
The $\eta_0$-adapted cross-section in the above proposition can be built from any cross section that contains $x$ in its interior. 
\end{remark}

Now, we need to consider a special kind of neighborhoods of the singularities contained in an ASH attractor. For this, we recall the construction given in \cite{SYY}: Let $\beta_1>0$ such that
\begin{enumerate}[(a)]
    \item $B_{\beta_1}(\sigma)\cap B_{\beta_1}(\sigma')=\emptyset$, where $B_r(a)$ denotes the open ball centered in $a$ and radius $r>0$ and $\sigma,\sigma'\in Sing_{\Lambda}(X)=Sing(X)\cap\Lambda$.
    \item The map $exp_{\sigma}$ is well defined on $\lbrace v\in TM_{\sigma} : \Vert v\Vert\leq\beta_1\rbrace$ for every $\sigma\in Sing_{\Lambda}(X)$. 
    \item There are $L_0, L_1>0$ such that 
    \begin{displaymath}
    L_0\leq\frac{\Vert X(x)\Vert}{d(x,\sigma)}\leq L_1,\quad\forall x\in \overline{B_{\beta_1}(\sigma)},\quad\forall \sigma\in Sing_{\Lambda}(X). 
    \end{displaymath}
    \item The flow in $B_{\beta_1}(\sigma)$ is a small $C^1$ perturbation of the linear flow. 
\end{enumerate}
For every $\sigma\in Sing_{\Lambda}(X)$ define
\begin{displaymath}
D_{\sigma}=exp_{\sigma}(\lbrace v=(v^s,v^u)\in TM_{\sigma} : \Vert v\Vert\leq\beta_1, \Vert v^s\Vert=\Vert v^u\Vert \rbrace)\subset M,
\end{displaymath}
and 
\begin{displaymath}
D_n=D_{\sigma}\cap(B_{e^{-n}}(\sigma)\setminus B_{e^{-(n+1)}}(\sigma)),\quad\forall n\geq n_0,
\end{displaymath}
where $n_0$ is large enough. 
In \cite{SYY} the authors constructed a partition of the  cross sections $\Sigma_{\sigma}^{i,o,\pm}$ given in \cite{AP}. More precisely, assume that $\Sigma_{\sigma}^{i,o,\pm}\subset\partial B_{\beta_1}(\sigma)$. Let consider 
\begin{displaymath}
D_n^o=\bigcup_{x\in D_n}X_{t_x^+}(x), \quad D_n^i=\bigcup_{x\in D_n}X_{-t_x^-}(x),\quad\forall n\geq n_0,
\end{displaymath}
where 
\begin{displaymath}
t_x^+=\inf\lbrace\tau>0 : X_{\tau}(x)\in \Sigma_{\sigma}^{o,\pm} \rbrace,
\end{displaymath}
and
\begin{displaymath}
t_x^-=\inf\lbrace\tau>0 : X_{-\tau}(x)\in \Sigma_{\sigma}^{i,\pm} \rbrace.
\end{displaymath}
Note that $\lbrace D_n^i\cap\Sigma_{\sigma}^{i,\pm}\rbrace_{n\geq n_0}$ form a partition of $\Sigma_{\sigma}^{i,\pm}$ for which $X_{\tau(x)}(x)\in D_n^o\cap\Sigma_{\sigma}^{o,\pm}$ for every  $x\in D_n^i\cap\Sigma_{\sigma}^{i,\pm}$ and every $n\geq n_0$, where $\tau(\cdot)$ is the flight time to go from $\Sigma_{\sigma}^{i,\pm}$ to $\Sigma_{\sigma}^{o,\pm}$. Moreover, they showed that this flight time satisfies 
\begin{equation}\label{vuelo}
\tau(x)\approx \frac{\lambda_u+1}{\lambda_u}n,\quad \forall x\in D_n^i\cap\Sigma_{\sigma}^{i,\pm},\quad\forall n\geq n_0.
\end{equation}
In this case, let $\widetilde{\Sigma_{\sigma}^{i,\pm}}=\left(\left(\bigcup_{n\geq n_0}D_n^i\cap\Sigma_{\sigma}^{i,\pm}\right)\right)\cup \ell_{\pm}$, where $\ell_{\pm}$ is contained in $ W^s_{loc}(\sigma)\cap\Sigma_{\sigma}^{i,\pm}$, and let consider
\begin{equation}\label{vecindad}
V_{\sigma}=\bigcup_{z\in\widetilde{\Sigma_{\sigma}^{i,\pm}}\setminus\ell_{\pm}}X_{(0,\tau(z))}(z)\cup (-e^{-n_0},e^{-n_0})\times(-e^{-n_0},e^{-n_0})\times(-1,1)\subset U. 
\end{equation}

Denote $\widetilde{\Sigma^{i,o,\pm}}=\bigcup_{\sigma\in Sing_{\Lambda}(X)}\widetilde{\Sigma_{\sigma}^{i,o,\pm}}$ and $V=\bigcup_{\sigma\in Sing_{\Lambda}(X)}V_{\sigma}$.

\begin{remark}
We have the following remarks: 
\begin{itemize}
    \item By Remark 3.2 we can assume without loss of generality that every cross section in $\widetilde{\Sigma^{i,o,\pm}}$ is $\eta_0$ adapted for some $\eta_0>0$. 
    \item The above construction can be made by taking $\widetilde{\beta}<\beta_1$. In this case, we denote
    \begin{displaymath}
    \widetilde{\Sigma_{\widetilde{\beta}}^{i,o,\pm}}=\bigcup_{\sigma\in Sing_{\Lambda}(X)}\widetilde{\Sigma_{\sigma,{\widetilde{\beta}}}^{i,o,\pm}}\quad\text{and}\quad V_{\widetilde{\beta}}=\bigcup_{\sigma\in Sing_{\Lambda}(X)}V_{\sigma,{\widetilde{\beta}}}
    \end{displaymath}
\end{itemize}
\end{remark}

\begin{lemma}
Let $\widetilde{\varepsilon}_0>0$. There are a positive number $\delta_0$, independent of $\widetilde{\varepsilon}_0$, and a positive number $\widetilde{\beta}$ such that if $x,y\in\widetilde{\Sigma_{\widetilde{\beta}}^{i,\pm}}$ satisfy $d(X_t(x),X_{\theta(t)}(y))\leq\delta_0\Vert X(X_t(x))\Vert $ for every $t\in\mathbb{R}$, where $\theta:\mathbb{R}\to\mathbb{R}$ is a continuous and increasing function, then 
\begin{equation}\label{estimneighborhood}
    \frac{\vert \det DX_{\tau(y)}(y)\vert_{E^c_y}\vert}{\vert \det DX_{\tau(x)}(x)\vert_{E^c_x}\vert}\geq K_1C_0^{\tau(x)}, 
\end{equation}
where $C_0=C_0(\widetilde{\varepsilon}_0)$  and $K_1$ is a fixed positive constant. 
\end{lemma}
\begin{proof}
By uniform continuity of $DX_1(\cdot)$ and the subbundle $E^c$ on $\overline{U}$, we see that there is $\widetilde{\delta}_0>0$ such that 
\begin{equation}
    \text{if }\quad d(x,y)<\widetilde{\delta}_0,\quad\text{then}\quad \frac{\vert \det DX_1(y)\vert_{E^c_y}\vert}{\vert \det DX_1(x)\vert_{E^c_x}\vert}\geq C_0\approx 1,
\end{equation}
where $C_0=C(\widetilde{\varepsilon}_0)$. Besides, for such value of $\widetilde{\delta}_0$, there is $\widetilde{\beta}<\beta_1$ such that
\begin{equation}
    \Vert X(x)\Vert<\frac{L_0\widetilde{\delta}_0}{2},\quad\forall x\in \overline{B_{\widetilde{\beta}}(\sigma)},\quad\forall\sigma\in Sing_{\Lambda}(X).
\end{equation}

Let $0<\delta_0<\frac{e-1}{L_1}$. First, by part (c) and the choice of $\delta_0$ we have for every $s\in\mathbb{R}$ that 
\begin{eqnarray*}
d(X_s(x),X_{\theta(s)}(y))&\leq&\delta_0\Vert X(X_s(x))\Vert \nonumber\\
&\leq& \delta_0 L_1d(X_s(x),\sigma) \nonumber\\
&<&(e-1)d(X_s(x),\sigma).
\end{eqnarray*}
Second, by construction on $V_{\sigma}$ we see that there is $s>0$ such that $X_s(x)\in D_n(\sigma)$ for some $n\geq n_0$, i.e, $x$ belongs to $D_n^i\cap\Sigma_{\sigma,\widetilde{\beta}}^{i,\pm}$. Since $d(X_s(x),\sigma)<e^{-n}$, it follows that 
\begin{eqnarray}\label{cota1}
d(X_{\theta(s)}(y),\sigma)&\leq &d(X_s(x),X_{\theta(s)}(y))+d(X_s(x),\sigma) \nonumber \\
&\leq & ed(X_s(x),\sigma) \nonumber\\
&<&e^{-(n-1)}. 
\end{eqnarray}

On the other hand, we have 
\begin{eqnarray*}
e^{-(n+1)}\leq d(X_s(x),\sigma)&\leq&d(X_{\theta(s)}(y),\sigma)+d(X_s(x),X_{\theta(s)}(y)) \\
&\leq & d(X_{\theta(s)}(y),\sigma)+\delta_0 L_1d(X_s(x),\sigma).
\end{eqnarray*}
So, by the choice of $\delta_0$, 
\begin{equation}\label{cota2}
d(X_{\theta(s)}(y),\sigma)\geq (1-\delta_0 L_1)e^{-(n+1)}> e^{-(n+2)}.    
\end{equation}
Thus, if $\varepsilon_s>0$ satisfies $X_{v_s}(y)=X_{\theta(s)+\varepsilon_s}(y)\in D_{\sigma}$, we have by \eqref{cota1}  and \eqref{cota2} that 
\begin{displaymath}
X_{v_s}(y)\in D_{n-1}(\sigma)\cup D_{n}(\sigma)\cup D_{n+1}(\sigma). 
\end{displaymath}
This shows that 
\begin{displaymath}
y\in \left(D_{n-1}^i(\sigma)\cap\Sigma_{\sigma}^{i,\pm}\right)\cup \left(D_n^i(\sigma)\cap\Sigma_{\sigma}^{i,\pm}\right)\cup \left(D_{n+1}^i(\sigma)\cap\Sigma_{\sigma}^{i,\pm}\right). 
\end{displaymath}
So, 
\begin{equation}\label{L}
\vert\tau(x)-\tau(y)\vert\approx \frac{\lambda_u+1}{\lambda_u}=L.
\end{equation}

Now, note that if $x,y\in\Sigma_{\sigma,\widetilde{\beta}}^{i,\pm}$ and $n_0\in\mathbb{N}$ is the largest number such that $X_{n_0}(x),X_{n_0}(y)\in\overline{B_{\widetilde{\beta}}(\sigma)}$, then, by part (c) and (3.7), 
\begin{displaymath}
d(X_n(x),X_n(y))\leq \frac{1}{L_0}(\Vert X(X_n(x))\Vert+\Vert X(X_n(y))\Vert)\leq \widetilde{\delta}_0,\quad\forall n\leq n_0.
\end{displaymath}
By \eqref{L} we see that $\tau(y)=n_0+r_y$, $0\leq r_y<1$, and $\tau(x)=n_0+r_x$ with $\vert r_x\vert\leq L+1$. So, by (3.6) and the chain rule, we have by an easy computation the estimate \eqref{estimneighborhood}. 
\end{proof}

Now, let consider $\Lambda_+=\bigcap_{t\geq 0}X_{-t}(\Lambda\setminus V)$. 
\begin{lemma}
Given $\varepsilon_0>0$, there are positive numbers $\delta_1(\varepsilon_0), T_0$ and $K_2(\varepsilon_0)\approx 1$, and a neighborhood $W_0$ of $\Lambda_+$ such that for any $x,y\in W_0$ with $d(x,y)<\delta_1(\varepsilon_0)$, then 
\begin{equation}\label{estimateW}
    \frac{\vert\det DX_{t_1}(y)\vert_{E^c_y}\vert}{\vert\det DX_{t_1}(x)\vert_{E^c_x}\vert}\geq K_2(\varepsilon_0),\quad 0<t_1\leq T_0,  
\end{equation}
where $t_1\leq T_0$ is a first hyperbolic time for $x$ and  $y$. 
\end{lemma}
\begin{proof}
Let $x\in\Lambda_+$ and let $t_1=t_1(x)$ be the first $C$-hyperbolic time for $x$. By continuity of $E^c_{\Lambda}$ and $DX_{(\cdot)}(\cdot)$, there is $\delta_x>0$ such that such that $d(X_{t}(x),X_{t}(y))\leq \varepsilon_0$ and
\begin{displaymath}
    \vert \vert\det DX_{t_1}(x)\vert_{E^c_x}\vert-\vert \det DX_{t_1}(y)\vert_{E^c_y}\vert\vert<\varepsilon_0 
\end{displaymath}
for every $t\in[0,t_1]$ and $y\in B_{\delta_x}(x)$.  By the above inequality we deduce that
\begin{displaymath}
 \frac{\vert\det DX_{t_1}(y)\vert_{E^c_y}\vert}{\vert\det DX_{t_1}(x)\vert_{E^c_x}\vert}\geq K_2(x,\varepsilon_0),\quad K_2(x,\varepsilon_0)\approx 1.
\end{displaymath}
By compactness of $\Lambda_+$ we have that $\Lambda_+\subset \bigcup_{i=1}^mB_{\delta_{x_i}/2}(x_i)=W_0$. So, we can to find the numbers $\delta_1(\varepsilon_0), T_0$ and $K_2(\varepsilon_0)>0$ such that if $x,y\in W_0$ with $d(x,y)<\delta_1(\varepsilon_0)$, then $x,y\in B_{\delta_{x_i}}(x_i)$ for some $i\in\lbrace1,\ldots, m\rbrace$, so that \eqref{estimateW} is obtained by the finiteness of the open cover $W_0$. 
\end{proof}

Let consider the compact set $\Lambda''=\Lambda\setminus(V\cup W_0)$, where $W_0$ is as the above lemma.  
\begin{lemma}
Given $\varepsilon_1>0$, there are positive numbers $\delta_2(\varepsilon_1),T_1$ and $K_3(\varepsilon_1)\approx 1$, and a neighborhood $W_1$ of $\Lambda''$ such that for every $x,y\in W_1$, with $d(x,y)\leq \delta_2(\varepsilon_1)$, there is $0<s\leq T_1$ such that $X_s(x),X_s(y)\in V$ and  
\begin{equation}\label{estimateoutside}
    \frac{\vert\det DX_s(y)\vert_{E^c_y}\vert}{\vert\det DX_s(x)\vert_{E^c_x}\vert}\geq K_3(\varepsilon_1).
\end{equation}
\end{lemma}
\begin{proof}
Note that for $x\in\Lambda''$, there is $t_x>0$ such that $X_{t_x}(x)\in V$, otherwise we get that $x\in \Lambda_+\subset W_0$ which is impossible. So, by continuity of $X_{(\cdot)}(\cdot)$ and $DX_{(\cdot)}(\cdot)$, there are $\delta_x>0$ and $K_3(x,\varepsilon_1)\approx1$ such that $X_{t_x}(y)\in V$,  
\begin{displaymath}
    \vert \vert\det DX_{t}(x)\vert_{E^c_x}\vert-\vert \det DX_{t}(y)\vert_{E^c_y}\vert\vert<\varepsilon_1 
\end{displaymath}
and  
\begin{displaymath}
 \frac{\vert\det DX_{t_x}(y)\vert_{E^c_y}\vert}{\vert\det DX_{t_x}(x)\vert_{E^c_x}\vert}\geq K_3(x,\varepsilon_1), 
\end{displaymath}
for every $t\in[0,t_x]$ and $y\in B_{\delta_x}(x)$. In this way, we get an open cover of $\Lambda''$. So, the result is obtained by compactness of $\Lambda''$. 
\end{proof}

\begin{lemma}
Let $\widetilde{U}=\overline{U\setminus V}$, where $U$ is the basin of attraction of $\Lambda$. There are $\beta'>0$ and $\varepsilon_2>0$ such that if $x\in \widetilde{U}$, and $y,z\in B_{\varepsilon_2}(x)$, $z\in\mathcal{O}(y)$, then $z=X_u(y)$, $\vert u\vert<2\beta'$. 
\end{lemma}
\begin{proof}
Note that every point of $\widetilde{U}$ is regular and it is away from any singularity of $\Lambda$. Thus, for every $x\in\widetilde{U}$, it is possible to choice a small cross section $\Sigma_x$ containing $x$, with $\Sigma_x\cap Sing_{\Lambda}(X)=\emptyset$. Moreover, we can choose these cross sections in a such way that $B_{\alpha_0}(\Sigma_x)=\emptyset$ for a fixed $\alpha_0$ small enough. By uniform continuity of the flow in $\widetilde{U}$, there is $\beta'>0$ such that $X_{[-\beta',\beta']}(z)\in B_{\alpha_0}(\Sigma_x)$ for any $z\in \Sigma_x$. Let consider
\begin{displaymath}
 \widetilde{V}_x=\bigcup_{z\in\mathring{\Sigma}_x}X_{(-\beta',\beta')}(z),
\end{displaymath}
where $\mathring{A}$ denotes the interior of $A$, and let $D_x$ a compact subset of $\mathring{\Sigma_x}$ that contains $x$. Note that for every $z\in D_x$, there is $\varepsilon_z>0$ such that $B_{\varepsilon_z}(z)\subset\widetilde{V}_x$ because $\widetilde{V}_x$ is an open set. Since $D_x\subset\bigcup_{z\in D_x}B_{\varepsilon_z}(z)$ there is, by compactness of $D_x$, a positive number $\varepsilon_x$ such that $B_{\varepsilon_x}(D_x)\subset\widetilde{V}_x$. So, we have that 
\begin{displaymath}
\widetilde{U}\subset\bigcup_{x\in\widetilde{U}}B_{\varepsilon_x}(D_x). 
\end{displaymath}
Thus, by compactness of $\widetilde{U}$, there exists the Lebesgue number $\ell>0$. So, if define $\varepsilon_2=\ell/2$,  we obtain the desired result.    
\end{proof}
\begin{remark}
In the above lemma, is possible to choose $\beta'>0$ such that
\begin{equation}\label{estimatechico}
\beta'<\frac{\varepsilon}{2}\quad\text{and}\quad K_4(\beta')=\min_{(z,s)\in\widetilde{U}\times [-\beta',\beta']}\vert\det DX_{s}(z)\vert_{E^c_z}\vert\approx 1.  
\end{equation}
\end{remark}

In a similar way, we can obtain the next result: 
\begin{lemma}
Let $L>0$ as in \eqref{L}. Then, there is $\delta_3>0$ with the following property: For $x\in\Sigma_{\sigma}^{o,i,\pm}$, $\sigma\in Sing_{\Lambda}(X)$, and every $z\in M$ with $d(x,z)<\delta_3$, there exists $r\in\mathbb{R}$, with $\vert r\vert<L$, such that $X_r(z)\in\Sigma_{\sigma}^{o,i,\pm}$.
\end{lemma}

Now, by continuity of $DX_{(\cdot)}(\cdot)$, there is $a_0>0$ such that for every $z\in U$ and every two-dimensional subspace $\Pi_z\subset C_{a_0}^c(z)$ one has 
\begin{equation}\label{disto1}
\frac{\vert\det DX_t(z)\vert_{\Pi_z}\vert}{\vert\det DX_t(z)\vert_{E^c_z}\vert}\geq J_0(a_0)\approx 1,\quad\forall t\in[0,1].     
\end{equation}
Finally, by Theorem 2.2 every singularity $\sigma$ contained in an ASH attractor is either Lorenz-like or Rovella-like or resonant. In particular, there is a constant $J_1>0$ and numbers $\alpha_+>0$ and $\alpha_-<0$ such that for every $x\in V_{\sigma}$, where $V_{\sigma}$ is as in \eqref{vecindad}, and for every $t>0$ such that $X_t(x)\in V_{\sigma}$,
\begin{displaymath}
\vert \det DX_t(x)\vert_{E^c_x}\vert\geq
\begin{cases}
J_1e^{\alpha_{+}t}  &  \text{if  }\sigma\text{ is Lorenz -like}\\
J_1e^{\alpha_{-}t}   &  \text{if  }\sigma\text{ is Rovella -like or resonant}.
\end{cases}
\end{displaymath}
\begin{remark}
The above relation allows us to find a positive number $T_2$ and a constant $0<C_1<C$ such that if $t_x\geq T_2$ is a $C$-hyperbolic time for $x\in\Lambda'$, there is a $C_1$-hyperbolic time $t_x'>t_x>0$ satisfying $X_{t_x'}(x)\in\widetilde{\Sigma^{i,o,\pm}}$. 
\end{remark}

The previous lemmas and remarks are used to prove the following result: 
\begin{lemma}
There exists positive numbers $\delta_4$, $T$, $c_*$ such that given $x\in\Lambda'$, $y\in U$, and a continuous and increasing function $\theta:\mathbb{R}\to\mathbb{R}$ satisfy
\begin{displaymath}
d(X_t(x),X_{\theta(t)}(y))\leq\delta_4\Vert X(X_t(x))\Vert,\quad\forall t\in\mathbb{R},
\end{displaymath}
and given a $C$-hyperbolic time $t_x\geq T$, there is  $t_y>0$ such that $X_{t_y}(y)=X_{\theta(t_x)}(y)$ and 
\begin{equation}\label{expansion}
\vert\det DX_{t_y}(y)\vert_{E^c_y}\vert\geq e^{c_*t_x}.
\end{equation}
\end{lemma}
\begin{proof}
Let $0<C_1<C$ given by Remark 3.5. Take $\beta', \varepsilon_2>0$ as in Lemma 3.4, positive numbers $\widetilde{\varepsilon}_0$, $\varepsilon_0,\varepsilon_1<\varepsilon_2$, and positive constants $\alpha$, $c^*$, $a_0$ (as in \eqref{disto1}) satisfying
\begin{equation}\label{exptime}
    C_1-(\alpha+\vert\ln C_0\vert+\vert\ln J_0(a_0)\vert+ \vert\ln K_2(\varepsilon_0)\vert+\vert\ln K_3(\varepsilon_0)\vert+\vert\ln K_4(\beta')\vert)>c^*>0.
\end{equation}
Let $\delta_0$, $\delta_1(\varepsilon_0)$, $\delta_2(\varepsilon_1)$  and $\delta_3$ given by Lemma 3.1, Lemma 3.2, Lemma 3.3 and Lemma 3.5 respectively. Let consider
\begin{displaymath}
0<\delta_4<\min\left\lbrace \delta_0,\frac{\delta_1(\varepsilon_0)}{K_m},\frac{\delta_2(\varepsilon_1)}{K_m},\frac{\varepsilon_2}{K_m},\frac{\delta_3}{K_m}\right\rbrace,\quad K_m=\max_{z\in\Lambda}\Vert X(z)\Vert.
\end{displaymath}
Let $x\in\Lambda'$, $y\in U$, and $\theta$ satisfying the condition given in the statement of lemma. Let $s_0\geq 0$ (if it exists) the first flight time satisfying $x_0=X_{s_0}(x)\in \widetilde{\Sigma_{\widetilde{\beta}}^{i,\pm}}$, where $\widetilde{\beta}$ is given by Lemma 3.1. Then, by the choice of $\delta$, we have that 
\begin{displaymath}
d(X_{s_0}(x),X_{\theta(s_0)}(y))\leq\delta\Vert X(X_{s_0}(x))\Vert<\delta_3,
\end{displaymath}
and
\begin{displaymath}
d(X_{s_0+\tau(x_0)}(x),X_{\theta(s_0+\tau(x_0))}(y))\leq\delta\Vert X(X_{\tau(x_0)}(x_0))\Vert<\delta_3.
\end{displaymath}
Moreover, by Lemma 3.1 and Lemma 3.5,
\begin{equation}\label{tsing}
\theta(s_0+\tau(x_0)))=\theta(s_0)+\tau(y')+r_y,
\end{equation}
where $y'\in\widetilde{\Sigma^{i,\pm}}$, $\vert\tau(x_0)-\tau(y')\vert<L$ and $\vert r_y\vert< 2L$. Besides, by letting $\tau(y')=N(y)+t_{y}$, $t_{y}\in[0,1)$, we have by \eqref{disto1} that
 \begin{equation}\label{e1}
\vert \det DX_{\tau(y')+r_y}(X_{\theta(s_0)}(y))\vert_{\Pi_{X_{\theta(s_0)}(y)}}\vert\geq J_0(a_0)^{N(y')+1}K' C_0^{\tau(x_0)}\vert \det DX_{\tau(x_0)}(x_0)\vert_{E^c_{x_0}}\vert,
\end{equation}
where $K'=K_0^2K_1$ and $K_0=\min_{(z,s)\in\widetilde{U}\times [-L,L]}\vert\det DX_{s}(z)\vert_{E^c_z}\vert$. 

Now, according to the choice of $\delta$ we see that if $x_1=X_{s_0+\tau(x_0)}(x)$ and $y_1=X_{\theta(s_0+\tau(x_0))}(y)$ belongs to $W_0$, then, by Lemma 3.2 and Lemma 3.4 we have that
\begin{equation}\label{t1}
    \theta(s_0+\tau(x_0)+t_1)=\theta(s_0+\tau(x_0))+t_1+r_{y_1},\quad \vert r_{y_1}\vert<2\beta',
\end{equation}
where $t_1=t_1(x_1)\leq T_0$ is a first $C$-hyperbolic time for $x_1$, and by \eqref{disto1}
\begin{equation}\label{e2}
     \vert\det DX_{t_1+r_{y_1}}(y_1)\vert_{\Pi_{y_1}}\vert\geq J_0(a_0)^{N(y_1)+1}K_4(\beta')K_2(\varepsilon_0)\vert\det DX_{t_1}(x_1)\vert_{E^c_{x_1}}\vert,
\end{equation}
where $N(y_1)$ satisfies $t_1=N(y_1)+r_{y_1}'$, $r_{y_1}'\in[0,1)$, whereas if $x_1$ and $y_1$ belongs to $W_1$, by Lemma 3.2 we have that $X_{s_1}(x_1)\in \widetilde{\Sigma_{\widetilde{\beta}}^{i,\pm}}$, $0\leq s_1=s_1(x_1)\leq T_1$. Moreover, 
\begin{equation}\label{t2}
    \theta(s_0+\tau(x_0)+s_1)=\theta(s_0+\tau(x_0))+s_1+r_{y_1},\quad \vert r_{y_1}\vert<2\beta',
\end{equation}
and by \eqref{disto1}
\begin{equation}\label{e3}
    \vert\det DX_{s_1+r_{y_1}}(y_1)\vert_{E^c_{y_1}}\vert\geq J_0(a_0)^{N(y_1)+1}K_4(\beta')K_3(\varepsilon_0)\vert\det DX_{s_1}(x_1)\vert_{\Pi_{x_1}}\vert,
\end{equation}
where $N'(y_1)$ satisfies $s_1=N'(y_1)+r_{y_1}'$, $r_{y_1}'\in[0,1)$. 

Let 
\begin{displaymath}
x_n =
\begin{cases}
X_{t_{n-1}}(x_{n-1})  &  \text{if  }x_{n-1}\in W_0\\
X_{s_{n-1}}(x_{n-1})   &  \text{if  }x_{x_{n-1}}\in W_1
\end{cases}, 
n\geq 2,
\end{displaymath} 
and let consider the following sets: 
\begin{itemize}
\item $A=\lbrace n\in\mathbb{N}: x_n\in W_0\rbrace$, $n_A=\#A$
\item $B=\lbrace n\in\mathbb{N} : x_n\in W_1\rbrace$, $n_B=\#B$, and 
\item $O=\lbrace n\in\mathbb{N} : x_n\in\widetilde{\Sigma_{\widetilde{\beta}}^{i,\pm}}\setminus(W_0\cup W_1) \rbrace$, $n_O=\#O$.
\end{itemize}
Define 
\begin{eqnarray*}
t_x'&=&s_0+\sum_{m\in A}(t_n)+\sum_{n\in B}(s_n)+\sum_{n\in O}(\tau_n),\\
t_y'&=&\theta(s_0)+\sum_{n\in A}(t_n+r_{y_n})+\sum_{n\in B}(s_n+r_{y_n})+\sum_{n\in O}(\tau_n+r_{y_n}),
\end{eqnarray*}
and denote $\varphi_t(z)=\vert \det DX_{t}(z)\vert_{E_x^c}\vert$. Note that every $C$-hyperbolic time $t_x$ can be written as $t_x=t_x'+r_x$, with $r_x\in[0,T_m)$, where $T_m=\max\lbrace T_0,T_1 \rbrace$. In this case, define $t_y=t_y'+r_x+s_y$, where $\vert s_y\vert<2\beta'$.

The relations \eqref{tsing} and \eqref{t2} show that $X_{t_y}=X_{\theta(t_x)}$, and by the estimations \eqref{e1}, \eqref{e2} and \eqref{e3} we have that 
\begin{eqnarray*}
\vert\det DX_{t_y}(y)\vert_{E^c_y}\vert&=&\rho\prod_{n\in A}\varphi_{t_n+r_{y_n}}(x_n)\prod_{n\in B}\varphi_{s_n+r_{y_n}}(x_n)\prod_{n\in O}\varphi_{\tau_n+r_{y_n}}(x_n) \\
&\geq&\rho J_0(a_0)^{N_y}\cdot K_2(\varepsilon_0)^{n_A}\cdot K_3(\varepsilon_0)^{n_B}\cdot \\
&& K_4(\beta')^{n_A+n_B}\cdot (K')^{n_O}\cdot C_0^{\sum_{n\in O}\tau(x_n)}\varphi_{t_x}(x),
\end{eqnarray*}
where 
\begin{displaymath}
N_y=\sum_{n}N(y_n)+(n_A+n_B+n_O). 
\end{displaymath}
Since $t_x$ is a $C$-hyperbolic time for $x$, by the above estimate, we deduce that 
\begin{equation}\label{final}
\vert\det DX_{t_y}(y)\vert_{E^c_y}\vert\geq \rho e^{\left(C_1+N(C_0)+N(K')+N(J_0)+ N(K_2)+N(K_3)+N(K_4)\right)t_x},
\end{equation}
where 
\begin{displaymath}
N(C_0)=\left(\frac{\sum_{n\in O}\tau(x_n)}{t_x}\right)\ln C_0,\quad N(K')=\frac{n_O}{t_x}\ln K',
\end{displaymath}
\begin{displaymath}
N(J_0)=\frac{N_y}{t_x}\ln J_0(a_0),\quad N(K_2)=n_A\frac{\ln K_2(\varepsilon_0)}{t_x},
\end{displaymath}
and
\begin{displaymath}
N(K_3)=n_B\frac{\ln K_3(\varepsilon_0)}{t_x},\quad N(K_4)=(n_A+n_B)\frac{\ln K_4(\beta')}{t_x}.
\end{displaymath} 
First, since $t_x>\sum_{n\in O}\tau(x_n)$ it follows that 
\begin{displaymath}
\vert N(C_0)\vert \leq \vert \ln C_0\vert. 
\end{displaymath}
Second, by \eqref{vuelo},
\begin{displaymath}
\left\vert\frac{n_O}{t_x}\ln K'\right\vert\leq \frac{\vert\ln K'\vert}{n_0},
\end{displaymath}
so that, by shrinking $V_{\widetilde{\beta}}$ if it is necessary, we have that \begin{displaymath}
\left\vert\frac{n_O}{t_x}\ln K'\right\vert\leq\alpha.
\end{displaymath}
On the other hand, since $(n_A+n_B+n_O)/t_x\leq 1$ and $N_y/t_x\leq 2$ we obtain the following estimations:
\begin{itemize}
    \item $\vert N(J_0)\vert\leq\vert\ln J_0(a_0)\vert$,
    \item $\vert N(K_2)\vert\leq \vert\ln K_2(\varepsilon_0)\vert$, 
    \item $\vert N(K_3)\vert\leq \vert\ln K_3(\varepsilon_0)\vert$ and 
    \item $\vert N(K_4)\vert\leq \vert\ln K_4(\beta')\vert$.
\end{itemize}
So, by \eqref{exptime} and \eqref{final} it follows that 
\begin{eqnarray*}
\vert\det DX_{t_y}(y)\vert_{E^c_y}\vert&\geq& \rho e^{\left(C_1-(\alpha+\vert\ln C_0\vert+\vert\ln J_0(a_0)\vert+ \vert\ln K_2(\varepsilon_0)\vert+\vert\ln K_3(\varepsilon_0)\vert+\vert\ln K_4(\beta)\vert)\right)t_x}\\
&\geq &\rho e^{c^*t_x}.
\end{eqnarray*}

Finally, if $T>0$ satisfies $\rho e^{c^*t}\geq e^{c_*t}$, $0<c_*<c^*$, for any $t\geq T$ we have, for every $C$-hyperbolic time $t_x\geq T$, that 
\begin{displaymath}
\vert\det DX_{t_y}(y)\vert_{E^c_y}\vert\geq \rho e^{c_*t_x}\geq e^{c_*t_x}.
\end{displaymath}
This concludes the proof.
\end{proof}

Now, the following construction can be seen as a rescaled version of the tube-like domain introduced in \cite{AP}. For a regular point $x\in M$, define
\begin{displaymath}
N^r_{\beta}(x)=exp_x(\lbrace v\in T_xM : v\perp X(x)\rbrace\cap B_{\beta ||X(x)||}(0)),
\end{displaymath}
where $B_r(0)$ is the open ball in $T_xM$ of radius $r$ and centered at $0$. 

\begin{lemma}[\cite{WW}]\label{RFB}
Suppose that $X$ is a $C^1$-vector field and let $X_t$ be the flow induced by $X$. Then, there exist $L>0$ and a small $\beta_0>0$ such that for any $0<\beta<\beta_0$, $t>0$ and $x\in M \setminus Sing(X)$ we have:
 \begin{enumerate}
\item The set $X_{[-\beta\Vert X(x)\Vert,\beta\Vert X(x)\Vert]}(N^r_{\beta}(x))$ is a flow box, in particular it does not contain singularities.
\item The open ball $B_{\beta\Vert X(x)\Vert}(x)$ is contained on  $X_{[-\beta\Vert X(x)\Vert,\beta\Vert X(x)\Vert]}(N^r_{\beta}(x))$.

\item The holonomy map   $$P_{x,t}:N_{\frac{\beta}{L^{t}}\Vert X(x)\Vert}^r(X) \to N^r_{\beta\Vert X(X_t(x))\Vert}(X_t(x))$$ is well defined and injective. Moreover, for any $y\in N_{\frac{\beta}{L^{t}}\Vert X(x)\Vert}^r(X)$ we have $$d(X_s(x),X_s(y))\leq \beta\Vert X(X_s(x))\Vert$$ for any $0\leq s\leq t$. The same statement is valid for $t<0$.  

\end{enumerate}

\end{lemma}

    Let $\Lambda\subset M$ be an ASH attractor, with splitting $E^s\oplus E^c$ on $U$. Then, the stable manifold theorem says that every point in $U$ has immersed strong-stable manifold and  local strong-stable manifold which will be denoted as $W^{ss}(x)$ and $W^{ss}_{\varepsilon'}(x)$, $0<\varepsilon'<1$, respectively. In particular, given a regular point $x$, we denote its stable manifold and local stable manifold, respectively, by the sets  $$W^s(x)=\bigcup_{t\in \mathbb{R}}W^{ss}(X_t(x)) \textrm { and } W^s_{\varepsilon'}(x)=\bigcup_{t\in\mathbb{R}}W_{\varepsilon'}^{ss}(X_t(x)).$$   
Let $\eta>0$ be such that the exponential map $exp_x$ is injective for any $x\in M$. Fix $$0<\delta'\leq \min\left\lbrace\frac{\min\{\eta,\beta_0\}}{B},\varepsilon',\eta,\beta_0\right\rbrace,$$ where $$B=\sup\limits_{x\in M}\{\Vert X(x)\Vert\}.$$ 
By Lemma \ref{RFB}, we have that $N^r_{\delta'}(x)$ is a cross-section of time $\delta'\Vert X(x)\Vert$ for any regular point $x\in M$. Now fix $T>0$, $\delta''=\frac{\delta'}{L^T}$ and let us define for any regular point $x$ the set   $W^{r,s}_{\delta''}(x)$ to be the connected component that contains $x$ in the set $$W^s_{\varepsilon'}(x)\cap N^r_{\delta''}(x).$$

 Suppose that  $$d(X_t(x),X_{\theta(t)}(y))\leq \delta''\Vert X(X_t(x))\Vert,\quad \forall t\in\mathbb{R},$$ where $\theta$ is an increasing homeomorphism. Notice that this forbids $y\in W^{ss}_{\varepsilon'}(x)$ due to the exponential expansion of $W^{ss}_{\alpha}(x)$ in the past.  In addition, suppose that $y\notin \mathcal{O}(x)$. By the  Lemma \ref{RFB}, we can assume that   $X_{\theta(iT)}(y)$ belongs to $N^r_{\delta''}(X_{iT}(x))$, for any $i\in \mathbb{Z}$, unless doing a modification in $\theta$. Since $\delta''\Vert X(X_t(x))\Vert\leq \eta$ for any $t\in [0,T]$ one can fix a curve $\gamma^0_t:[0,1]\to N^r_{\delta''}(X_{t}(x))$ satisfying:
 
 \begin{itemize}
     \item $\gamma^0_t(0)=X_t(x)$ and $\gamma^0_t(1)=X_{\theta(t)}(y)$
     \item $\gamma^0_t(s)$ is transversal to  $W^{r,s}_{\delta''}(\gamma^0_t(s))$ for any $s\in [0,1]$.
     \item $d(\gamma^0_t(s),X_t(x))\leq \delta''\Vert X(X_t(x))\Vert$, for any $s\in [0,1]$.
     \item For any $s\in [0,1]$, the function $\gamma^0_t(s)$ varies continuously wit $t$. 
 \end{itemize} 
 This generates a smooth immersion $\rho^0:[0,T]\times[0,1]\to M$ defined by $\rho(t,s)=\gamma_t(s)$. Similarly, define for each $i\in \mathbb{Z}\setminus \{0\}$ an smooth immersion $\rho^i:[0,T]\times[0,1]\to M$ such that:

\begin{itemize}
    \item For each $i$, if we fix $t$, $\rho^{i}(t,s)=\gamma^i_t(s)$ is a curve transversal to $W^{r,s}_{\varepsilon'}(\gamma^i_{t}(s))$
    \item $\gamma^i_t(0)=X_{iT+t}(x)$ and $\gamma^i_t(1)=X_{\theta(iT+t)}(y)$
    \item $d(\gamma_t^i(s),X_{iT+t}(x))\leq \delta'' \Vert X(X_{iT+t}(x))\Vert$, for any $s\in [0,1]$.
    \item $\rho:\mathbb{R}\times [0,1]\to M$ defined by $\rho(t,s)=\rho^i(t',s)$ if $t=iT+t'$ is a smooth immersion.
\end{itemize}

Now that we have the smooth immersed surface $S=Img(\rho)$, we can define the $R$-tube-like domain. Fix some $t\in \mathbb{R}$ and let us denote $\gamma_t(s)=\rho(t,s)$. Recall that for any $t\in \mathbb{R}$, the curve $\gamma_t(s)$ is transversal to $W^{r,s}_{\varepsilon'}(\gamma_t(s))$, thus we define the  $R$-\textit{tube-like domain} for the surface $S$ as the set
\begin{displaymath}
T^r=\bigcup_{t\in \mathbb{R}, s\in[0,1]}W^{r,s}_{\varepsilon}(\gamma_t(s)),
\end{displaymath}
where $\varepsilon>0$ is as in the beginning of this section. 

\textbf{ Claim:} $T^r$ does not contains singularities.

This is an immediate consequence of Lemma \ref{RFB}. Indeed, since $\delta''\leq\beta_0$, then we have that $N^r_{\delta''}(X_t(x))$ does not contain singularities for any $t\in \mathbb{R}$. Since, any point of $T^r$ belongs to some of these $R$-cross-sections, then $T^r$ does not contains singularities.

As in \cite{AP} it is possible to show the following main properties: 
\begin{enumerate}
    \item The intersection of the $R$-tube $T^r$ with a cross section $\Sigma_{\sigma}^{i,\pm}$ belongs to one of the connected components of $\Sigma_{\sigma}^{i,\pm}\setminus W^s(Sing(X))$.
    \item The positive orbit of any point of $T^r$ remains inside the tube. 
\end{enumerate}
With these properties in mind we can state and prove the next lemma:
\begin{lemma}
Let $\gamma=\gamma_0^0(s)$, $s\in[0,1]$. There is $\widetilde{\delta}>0$ such that for every point $z\in\gamma$ there exists an increasing continuous function $\theta_z:\mathbb{R}^+\to\mathbb{R}$ such that 
\begin{displaymath}
d(X_t(x),X_{\theta_z(t)}(z))\leq \widetilde{\delta},\quad\forall t\geq 0. 
\end{displaymath}
\end{lemma}
\begin{proof}
We begin fixing some constants. Fix $K=\max_{x\in \gamma}{\Vert X(x)\Vert}$.  Let us take $T>0$ and $0<\lambda\leq 1$  such that for any $z\in W^{ss}_{loc}(p)$  and $p\in M$, we have: $$d(X_{iT}(z),X_{iT}(p))\leq \lambda^{t}d(z,p) \textrm{ for any } t\in [0,T].$$ 
Let us set $\widetilde{\delta}=\frac{\delta' K }{1-\lambda}$   and take some $z\in \gamma$. Recall that by the construction of the R-tube-like domain we have that $X_T(z)\in B_{\delta'}^r(X_{T}(x))$. 

Hereafter we will consider a $R$-tube-like domain through the orbit of $x$ with $T>1$. Thus we are able to find  some $t_1>0$ such that $z_1=X_{t_1}(z)\in N^r_{\delta}(X_{T}(x))$.
On the other hand, $z_1\in W^{ss}_{loc}(p_1)$ for some $p\in \gamma_{0}^1(s)=\gamma^1$. But, by construction of $\gamma_0^{1}(s)$, we have that $$d(X_1,p_1)\leq \delta''\Vert X(X_T(x))\Vert.$$

Now, continuing the previous process we see that there is $t_2>0$  such that $X_{t_2}(p)\in N^r_{\delta''}(X_{2T}(x))$, but this implies that
\begin{eqnarray*}
d(X_{t_2}(z_1),X_{2T}(x))&\leq& d(X_{t_2}(z_1),X_{t_2}(p))+ d(X_{t_2}(p),X_{2T}(x)\\
&\leq&\delta'\Vert X_{T}(x)\Vert\lambda^T + \delta'\Vert X_{2T}(x))\Vert.
\end{eqnarray*}
Inductively, for any $i>0$ we can find a time $t_i>0$ and a point $p_i\in \gamma_0^1$ such that $z_i=X_{s_i}(z)\in W^{ss}_{loc}(p_i)$, where $s=\sum_{j=1}^{i}t_j$, and satisfying:

$$d(X_{s_i}(z),X_{iT}(x))\leq \sum_{j=1}^{i-1}\delta'\Vert X_{jT}(x)\Vert\lambda^{(i-j)T}\leq \frac{\delta' K}{1-\lambda^T}\leq \widetilde{\delta}.$$       
Finally, defining $\theta_z:\mathbb{R}^+\to \mathbb{R}$ as $\theta_z(iT)=s_i$ and linear from the intervals $[i,i+1]$ to the intervals $[s_i,s_{i+1}]$, we get the desired function $\theta_z$.
\end{proof}

\begin{remark}
We can take $\delta'$ so that $\widetilde{\delta}<\min\left\lbrace \delta_0,\delta_1(\varepsilon_0),\delta_2(\varepsilon_1),\varepsilon_2,\delta_3\right\rbrace$. 
\end{remark}

Next, we are ready to prove the Theorem 2.4.

\begin{proof}[Proof of Theorem \ref{TeoP}]
Let $\varepsilon>0$ and $\delta_n\to 0$, $x_n, y_n\in\Lambda$ and $\theta_n:\mathbb{R}\to\mathbb{R}$ as in the beginning of this section, i.e, satisfying the relations \eqref{assumption1} and \eqref{assumption2}. As in \cite{AP}, there exits a regular point $z\in\Lambda$ and $z_n\in\omega(x_n)$ such that $z_n\to z$. Let $\Sigma_{\eta}$ an $\eta$-adapted cross section containing $z$ in its interior. Thus, the positive orbit of $x_n$ intersects $\Sigma_{\eta}$ infinitely many times for $n$ large enough. Besides, by using flow boxes in a small neighborhood of $\Sigma_{\eta}\cup\widetilde{\Sigma^{i,o,\pm}}$ we can to find positive numbers $\delta_5, t_0$ such that for any $\Sigma'\subset \Sigma_{\eta}\cup\widetilde{\Sigma^{i,o,\pm}}$, $z\in\Sigma'$ and $w\in M$ with $d(z,w)<\delta_5$, there is $t_w\leq t_0$ such that $w'=X_{t_w}(w)\in\Sigma'$ and $d_{\Sigma'}(z,w')<K'\delta_5$, where $d_{\Sigma'}$ is the intrinsic distance in $\Sigma'$, for some constant positive constant $K'$ which depends on $\Sigma_{\eta}\cup\widetilde{\Sigma^{i,o,\pm}}$.  

Now, fix $N\in\mathbb{N}$ such that $0<\delta_N<\min\left\lbrace\delta_4,\frac{\widetilde{\delta}}{K_m},\eta,\eta_0,\frac{\delta_5}{K}\right\rbrace$, where $\eta_0$, $\delta_4$ and $\widetilde{\delta}$ are given by Remark 3.3, Lemma 3.6 and Lemma 3.8 respectively. In this case, let consider $x=x_N$, $y=y_N$ and $\theta=\theta_N$. 

\textbf{Claim: }There is $s\in\mathbb{R}$ such that $X_{\theta(s)}(y)\in W_{\varepsilon}^{ss}(X_{[s-\varepsilon,s+\varepsilon]}(x))$.

Assume that $X_{\theta(s)}(y)\notin W_{\varepsilon}^{ss}(X_{[s-\varepsilon,s+\varepsilon]}(x))$ for any $s\in\mathbb{R}$. First, note that $y\notin\mathcal{O}(x)$, otherwise by Lemma 3.4 and Remark 3.4 we obtain that $X_{\theta(0)}(y)$ belongs to  $X_{[-\varepsilon,\varepsilon]}(x)$, contradicting \eqref{assumption2}. Second, by changing slightly the points $x$, $y$ and the function $\theta$, and by the choice of $\delta$, we can to fix a cross section $\Sigma'\subset\Sigma_{\eta}\cup\widetilde{\Sigma^{i,o,\pm}}$ such that $x\in\Sigma'$ and to find a sequence of arbitrarily large $C$-hyperbolic times $\lbrace t_n\rbrace_{n\geq 1}$ of $x$ such that $x_n=X_{t_n+r_n}(x)\in\Sigma'$, where $0\leq r_n<T_m$, with $T_m=\max\lbrace T_0,T_1 \rbrace$. Besides, we have that $y_n=X_{\theta(t_n+r_n)+v_n}(y)\in\Sigma'$, where $\vert v_n\vert\leq t_0$. Moreover, by shrinking $U$ if it is necessary, by Lemma 2.7 in \cite{AP} there is $\kappa_0>0$ such that 
\begin{equation}\label{cotabajo}
\ell(\gamma_n)\leq\kappa_0d_{\Sigma'}(x_n,y_n)\leq \kappa_0K'\delta_5,    
\end{equation}
where $\gamma_n$ is any curve joining $x_n$ and $y_n$, for every $n\geq 1$. 
According to the choice of $\delta$, we can to define a curve $\gamma'\subset\Sigma'$ whose points $z'$ are given by $z'=X_{v_z}(z)$, $\vert v_z\vert\leq t_0$, for $z\in\gamma$, where $\gamma$ is as in Lemma 3.8.

On the other hand, take
\begin{displaymath}
\kappa=\min_{(z,s)\in \overline{B_{\delta_5}(\Sigma')}\times [-t_0,t_0]}\vert\det DX_{s}(z)\vert_{E^c_z}\vert\cdot\min_{(z,s)\in\overline{U}\times [0,T_m]}\vert\det DX_{s}(z)\vert_{E^c_z}\vert>0.
\end{displaymath}
Let $\lambda>\kappa_0K'\delta_5$, and let $n_1$ large enough such that $\kappa e^{c_*t_{n_1}}>\lambda$. By the construction of the $R$-tube $T^r$, Lemma 3.8 and following the proof of Theorem A in \cite{AP} it is possible to construct a Poincar\'e return map $R$ define on the whole strip between the stable manifolds of $x$ and $y'$ inside $\Sigma'$ (the stable manifolds in $\Sigma'$ are defined by $W^s(z,\Sigma')=W_{loc}^s(z)\cap\Sigma'$, $z\in\Sigma'$), whose return time $s(\cdot)$ satisfies 
\begin{displaymath}
s(x)\approx t_{n_2}+T_m\quad s(z')\approx t_0+\theta_z(t_{n_2})+T_m,\quad\forall z'\in\gamma', 
\end{displaymath}
where $n_2>n_1$ is large enough. The Figure 1 illustrates the construction of this map. 
\begin{figure}[ht]
\includegraphics[scale=0.3]{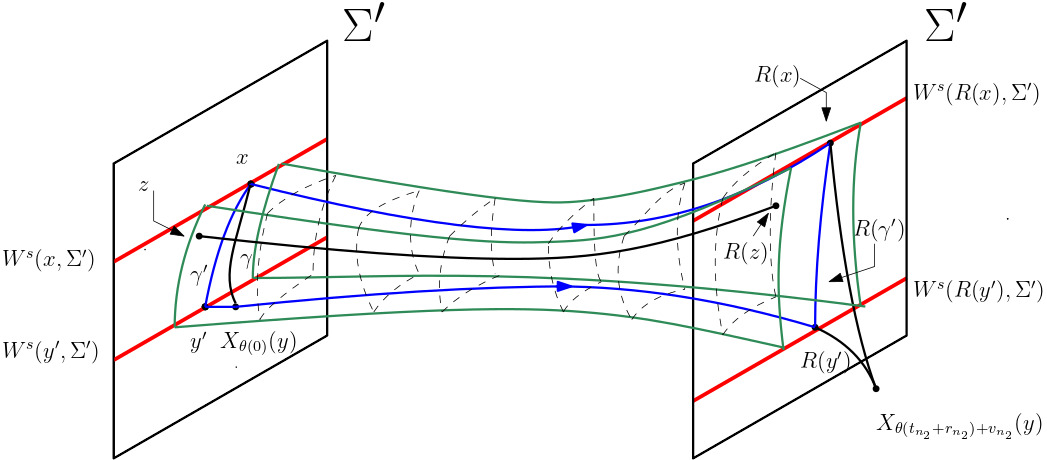}
\caption{The Poincar\'e map $R$.}
\end{figure}

By Lemma 3.8, Remark 3.6, and by following the proof of Lemma 3.6 we deduce, by shrinking $\varepsilon'$ if it is necessary, that the relation \eqref{expansion} is satisfied for any $z\in\gamma'$. So, by definition of $R$ we have that $R(\gamma')$ is a curve in $\Sigma'$ that connects $x_{n_2}$ with $y_{n_2}$ and satisfies
\begin{displaymath}
\ell(R(\gamma'))\geq ke^{c_*t_{n_1}}\geq \lambda> \kappa_0K'\delta_5,
\end{displaymath}
which contradicts \eqref{cotabajo}. So, we have that $y_{n_2}\in W^s(x_{n_2},\Sigma')$. Therefore, we deduce the claim by following step by step the argument given in \cite{AP}, p. 2456.   

The last part of the proof is entirely analogous to the showed by the authors in \cite{AP}, p. 2449. Here, we will briefly explain the main idea of their argument. Recall that since the bundle $E^c$ contains the flow direction, then the angle between $X$ and $E^{ss}$ is bounded away from zero. As a consequence, we have the following lemma:

\begin{lemma}[Lemma 3.2 in \cite{AP}]
There exist $\rho>0$ small and $c>0$, depending only on the flow, such that if $z_1,z_2,z_3$ are points in $\Lambda$ satisfying $z_3\in X_{[-\rho.\rho]}(z_2)$ and $z_2\in W^{ss}_{\rho}(z_1)$, then 
 $$d(z_1,z_3)\geq c\cdot\max\{d(z_1,z_2),d(z_2,z_3)\}  $$
\end{lemma}

Now, fix $\delta_N<\min\left\lbrace\delta_4,\widetilde{\delta}/K_m,\eta,\eta_0,\delta_5/K,c\rho/B\right\rbrace$, where $B=\sup\limits_{x\in M}\{\Vert X(x)\Vert\}$. By the choice of $\delta_N$, by the previous claim we have $X_{\theta(s)}(y)\in W^{ss}_{\varepsilon}(X_{[s-\varepsilon,s+\varepsilon](x))})$.  Since we have backward expansion in the strong-stable local manifolds, we can find the smallest $\eta>0$ such that:

$$d(X_{h(s)+\tau -\eta}(y),X_{s+\tau-\theta}(x))=\rho \textrm{ or }  d(X_{s -\eta}(y),X_{s+\tau-\theta}(y))=\rho,$$
but this combined with Lemma 3.9 implies that $$d(X_{s+\tau-\eta}(x),X_{h(s+\tau-\eta)}(y))\geq c\rho>B\delta_N\geq \delta_N\Vert X(X_{s+\tau-\eta}(x))\Vert,$$
which is a contradiction to \eqref{assumption1}. This concludes the proof of Theorem \ref{TeoP}.
\end{proof}

Next example shows that the assumption of $\Lambda$ be an attractor is not superfluous. Indeed, it exhibits a flow with an attracting ASH set  which neither transitive nor $R$-expansive.

\begin{example}
We start the construction of the example by considering the sphere $S^2$ and the classical Morse-Smale flow $Y_t$ for which the north-pole $N$ is a sink and the south pole  $S$ is a source. If one consider $S^2$ as an embedded surface on $\mathbb{R}^3$, this flow can be obtained as the flow generated by the gradient field $Y$ of the height function of the points in $S^2$. It is clear that $Y_t$ is not transitive.

\textbf{Claim:} $Y_t$ is not $R$-expansive.

For this purpose, we will simplify the setting assuming some properties on $Y$, without loss of generality. First, let us call $U_N$ and $U_S$ the Hartman-Grobman neighborhoods of the the sink and the source, respectively. Let us consider that the vector field  is linear $Y(x)=\alpha x$, with $\alpha<0$ if $x\in U_N$ and $Y(x)=-\alpha x$ if $x\in U_S$. Thus we have that the expression $Y_t(x)=e^{\alpha t}x$ is holds both for $t\geq 0$ with $x\in U_N$ and  $x\in U_S$ with $t\leq 0$. In  both cases one has $\Vert Y(Y_t(x))\Vert=\vert\alpha\vert\Vert e^{\alpha t} x \Vert$ (notice that here we are seeing $U_S$ and $U_N$ as small neighborhoods of the origin of $\mathbb{R}^2$).

If we take some other point $y\in U_N$ such that $\Vert y\Vert=\Vert x\Vert$, then we have that $\Vert e^{\alpha t}y\Vert=\Vert e^{\alpha t}x\Vert$ for any $t\geq 0$.  Thus we have that 
\begin{displaymath}
d(Y_t(y),Y_t)(x)=d(x,y)e^{\alpha t}.
\end{displaymath}
Now, fix some $\varepsilon>0$ and chose such $y$ satisfying 
\begin{displaymath}
d(x,y)\leq \varepsilon \Vert Y(x)\Vert.
\end{displaymath}
This implies that 
\begin{displaymath}
d(Y_t(y),Y_t(x))\leq d(x,y)e^{\alpha t}\leq  \varepsilon \Vert Y(x)\Vert e^{\alpha t}=\varepsilon\Vert Y_t(x)\Vert.
\end{displaymath}
The same holds if $x\in U_S$ and $t\leq 0$.

Since $A=S^2\setminus (U_N \cup U_S)$ is a compact set, we have that 
\begin{displaymath}
K=\inf_{x\in A} \{\Vert Y(x)\Vert\}>0.
\end{displaymath}
Fix $x\in U_S$. By the continuity of $Y_t$, if we take $y$ as above close enough to $x$ then we can guarantee that 
\begin{displaymath}
d(Y_t(x),Y_t(y))\leq\varepsilon\Vert Y(Y_t(x))\Vert
\end{displaymath}
until both the orbits of $x$ and $y$ enter $U_N$. Fix some $t_0>0$ such that above estimate holds and $Y_{t_0}(x)\in U_N$. If $y$ is close enough to $x$, we can to find a flight time $T_0$ close to $t_0$ such that $Y_{T_0}(y)$ at same distance of $N$ as $X_{t_0}(x)$. Thus, defining $\theta:\mathbb{R}\to \mathbb{R}$  as $$\theta(t)=\begin{cases} t, & t\leq 0\\
\frac{T_0}{t_0}t, & t\in [0,t_0] \\
T_0+t, & t\geq t_0,
\end{cases}   $$
we have the following estimate:
\begin{displaymath}
d(Y_t(x),Y_{\theta(t)}(y))\leq \varepsilon \Vert Y(Y_t)(x)\Vert,\quad\forall t\in\mathbb{R}.
\end{displaymath} 
Therefore, $Y_t$ cannot be $R$-expansive since $\varepsilon$ was taken arbitrary. This proves the claim. 

Once we have proved the claim we are able to construct the example. Let us consider the set $M=S^2\times [-1,1]$ and fix some $\alpha'<\alpha<0$. Define the vector field $X(x,s)=(Y(x),\alpha' s)$ and consider the flow $X_t$ generated by $X$. Is is clear that $S^2\times \{0\}$ is compact and invariant set for $X_t$. In addition, since $$X_t(x,s)=(Y_t(x),e^{\alpha' s}),$$ it is easy to see that $S^2\times\{0\}$ is an attracting set. Moreover, one can easily show that $S^2\times\{0\}$ is an asymptotically sectional hyperbolic set, but it is not  $R$-expansive by the previous claim.

\end{example}

Now we proceed to prove Theorem \ref{TeoSens}. 

\begin{proof}[Proof of Theorem \ref{TeoSens}.]

The proof will be by contradiction. Suppose that $X_t$ is not sensitive to the initial conditions. More precisely, assume that for every $\delta>0$ there is $x\in M$ and a neighborhood $N_x$ of $x$ such that for every $y\in N_x$ one has 
\begin{equation}\label{sensitive}
d(X_t(x),X_t(y))\leq\delta,\quad\forall t\geq0. 
\end{equation}
First, by Theorem 2.2 the singularities $\sigma$ on $\Lambda$ are hyperbolic of saddle type. So, by the Grobman-Hartman Theorem there is a neighborhood $U_{\sigma}$ of $\sigma$, a neighborhood $U_0$ of $0\in T_{\sigma}M$ and a homeomorphism $h:U_{\sigma}\to U_0$ that conjugates $X_t$ with the linear flow $L_t$. Moreover, this homeomorphism is H\"older continuous with H\"older exponent depending on the eigenvalues of $DX(\sigma)$, i.e., there is $C(\sigma)>0$ and $\alpha(\sigma)>0$ such that 
\begin{equation}\label{Holder}
    \Vert h(x)-h(y)\Vert\leq Cd(x,y)^{\alpha},\quad\forall x,y\in U_{\sigma}. 
\end{equation}
By shrinking the neighborhoods $U_{\sigma}$ if it is necessary, we denote 
\begin{displaymath}
C=\max_{\sigma\in Sing(X)\cap\Lambda}\lbrace C(\sigma)\rbrace\text{ and }\alpha=\min_{\sigma\in Sing(X)\cap\Lambda}\lbrace \alpha(\sigma)\rbrace.
\end{displaymath}
Besides, inside $U_0$, 
\begin{displaymath}
W^{ss}(0)=\lbrace (a,b,c): b=c=0\rbrace,\quad W^s(0)=\lbrace (a,b,c): b=0\rbrace
\end{displaymath}
and 
\begin{displaymath}
W^u(0)=\lbrace (a,b,c): a=c=0\rbrace.
\end{displaymath}

On the other hand, by shrinking $\widetilde{\Sigma^{i,o,\pm}}$ if it is necessary, we can assume that $V_{\sigma}\subset U_{\sigma}$. In particular, we have that  $\Sigma_{\sigma}^{i,o,\pm}\subset U_{\sigma}$ for every $\sigma\in\Lambda$. Moreover, note that, since $h$ is a homeomorphism, the family $\mathcal{F}=\lbrace h(D_n^i\cap\Sigma_{\sigma}^{i,\pm})\rbrace_{n\geq n_0}$ is a partition of $h(\Sigma_{\sigma}^{i,\pm})\subset U_{0}$. Thus, since $d(x,W^s(\sigma))\leq K'e^{-(\lambda_u+1)n}$ for $x\in D_n^i\cap\Sigma_{\sigma}^{i,\pm}$ for some $K'>0$ (see \cite{SYY}) and  $h(W^s(\sigma))=W^s(0)$, we have by \eqref{Holder} that 
\begin{eqnarray*}
d(h(x),W^s(0))&\leq& Cd(x,W^s(\sigma))^{\alpha}\\
&\leq &CK'e^{-\alpha(1+\lambda_u)n} \\
&=& e^{\left(\frac{\ln (CK')}{n}-\alpha(1+\lambda_u)\right)n}. 
\end{eqnarray*}
So, if $0<\varepsilon_0<\alpha(1+\lambda_u)$ and $n_0\geq 1$ satisfies $\vert\ln (CK')/n\vert<\varepsilon_0$ for every $n\geq n_0$, we have that
\begin{displaymath}
d(h(x),W^s(0))\leq e^{\rho n},\quad \forall x\in D_n^i\cap\Sigma_{\sigma}^{i,\pm},
\end{displaymath}
where $\rho=\varepsilon_0-\alpha(1+\lambda_u)<0$. In fact, since $\mathcal{F}$ is a partition of $h(\Sigma_{\sigma}^{i,\pm})$ we deduce that 
\begin{equation}\label{part}
e^{\rho(n+1)}\leq d(h(x),W^s(0))\leq e^{\rho n},\quad \forall x\in D_n^i\cap\Sigma_{\sigma}^{i,\pm},\forall n\geq n_0. 
\end{equation}
So, if $h(\Sigma_{\sigma}^{o,\pm})=\Sigma_0^{\pm}=\lbrace p=(\pm1,b,c):\vert b\vert,\vert c\vert<1\rbrace\subset U_0$, we have by \eqref{part} that the fligh time $\tau(p)$, $p\in h(\Sigma_{\sigma}^{i,\pm})$, to go from $h(\Sigma_{\sigma}^{i,\pm})$ to $\Sigma_0^{\pm}$ satisfies 
\begin{equation}\label{timexit}
    -\frac{\rho}{\lambda_u}n\leq\tau(p)\leq -\frac{\rho}{\lambda_u}(n+1). 
\end{equation}

Now, by following the proof of Lemma 3.4 we deduce that given $\beta>0$ there is $\varepsilon>0$ such that
\begin{displaymath}
v\in C', w\in h(\Sigma_{\sigma}^{i,\pm})\text{ satisfying } \Vert v-w\Vert<\varepsilon\text{ then } L_u(v)\in h(\Sigma_{\sigma}^{i,\pm})\text{ with }u\in(-\beta,\beta).
\end{displaymath}
Take a compact neighborhood $C'$ of $V_0$ inside $U_0$. By uniform continuity of $h$ on $h^{-1}(C')$, there is $\delta>0$ such that $\Vert h(z)-h(w)\Vert<\varepsilon$ if $z,w\in h^{-1}(C')$ satisfies $d(z,w)<\delta$.

In this case, let consider 
\begin{displaymath}
0<\beta<-\rho/\lambda_u\quad\text{ and }\quad 0<\varepsilon<1-e^{\rho+\lambda_u\beta},
\end{displaymath}
and let consider $x\in\Lambda$ and $y\in N_x$ satisfying the relation \eqref{sensitive}. Let $t\geq 0$ (if it exists) such that $X_t(x)\in D_n^i\cap\Sigma_{\sigma}^{i,\pm}$ for some $\sigma\in Sing(X)\cap\Lambda$ and $n\geq n_0$. The choice of $\delta$ implies that there is $s\in\mathbb{R}$ such that $y'=X_{s+t}(y)\in \Sigma_{\sigma}^{i,\pm}$. Moreover, both points $y'$ and $X_t(x)$ belongs to the same connected component of $\Sigma_{\sigma}^{i,\pm}\setminus\ell_{\pm}$.  

\textbf{Claim:} $y'\in \left( D_{n-1}^i\cap\Sigma_{\sigma}^{i,\pm}\right)\cup \left( D_n^i\cap\Sigma_{\sigma}^{i,\pm}\right)\cup \left(D_{n+1}^i\cap\Sigma_{\sigma}^{i,\pm}\right)$. 

Indeed, assume that $y'\in D_{n+k}^i\cap\Sigma_{\sigma}^{i,\pm}$ for some $k>1$. By the choice of $\delta$ there is $u\in(-\beta,\beta)$ such that 
\begin{displaymath}
h(y')=(a_0,b_0,c_0)=(e^{\lambda_{ss}u}a,e^{\lambda_uu}b,e^{\lambda_su}c)\in h(D_{n+k}^i\cap\Sigma_{\sigma}^{i,\pm}).
\end{displaymath}
Moreover, assume that $a,c>0$. Since $h(x)\in D_n^i\cap\Sigma_{\sigma}^{i,\pm}$ we have by \eqref{part} and \eqref{timexit} that
\begin{eqnarray*}
\varepsilon>\Vert  L_{\tau(h(x))}(h(x))-L_{\tau(h(x))}(h(y))\Vert&\geq& 1-e^{\lambda_u\tau(h(x))}a\\
&=&1-e^{-\lambda_uu}e^{\lambda_u\tau(h(x))}a_0\\
&\geq& 1-e^{\rho+\lambda_u\beta},
\end{eqnarray*}
which contradicts the choice of $\varepsilon$.
Therefore, by the above claim, we obtain the relation \eqref{L} and, in consequence, an estimate similar to \eqref{estimneighborhood} for some constants $K_1$ and $C_0$. So, the rest of the argument is similar that of proof of Theorem 2.4.
\end{proof}





\end{document}